\documentclass[psamsfonts]{amsart}

\usepackage{amssymb,amsfonts, amsmath}
\usepackage{dsfont}
\usepackage[all,arc]{xy}
\usepackage{enumerate}
\usepackage{mathrsfs}

\usepackage{color}   
\usepackage{hyperref}
\hypersetup{
    colorlinks=true, 
    linktoc=all,     
    linkcolor=black,  
    citecolor=blue,
    urlcolor=blue
}

\newtheorem{thm}{Theorem}[section]

\newtheorem{prop}[thm]{Proposition}
\newtheorem{lemma}[thm]{Lemma}

\theoremstyle{definition}
\newtheorem{defn}[thm]{Definition}

\newtheorem{exmp}[thm]{Example}

\newtheorem{notn}[thm]{Notation}

\theoremstyle{remark}
\newtheorem{rem}[thm]{Remark}

\makeatletter
\let\c@equation\c@thm
\makeatother
\numberwithin{equation}{section}

\newcommand{\Z}{\mathbb{Z}}

\newcommand{\N}{\mathbb{N}}

\newcommand{\bop}{\bigoplus}

\newcommand{\surj}{\twoheadrightarrow}

\usepackage{mathabx,epsfig}

\title{Box Product of $C_p$-Mackey functors}

\author{Kaitlyn Loyd}

\date{\today}

\begin{document}

\begin{abstract}

Let $G$ be a finite group. In this paper, we provide an exposition of $G$-Mackey functors and a symmetric monoidal product on the category of Mackey functors called the box product. After computing several examples of box products for the case of $G=C_p$, we move to the heart of the paper, which is to find and classify all $C_p$-Mackey functors invertible for the box product.

\end{abstract}

\maketitle 

\tableofcontents

\section*{Introduction} 
The first notion of a $G$-Mackey functor was introduced by J.A. Green and later by Dress to provide a unified treatment of several constructions found in representation theory \cite{ADress}. As a simple example of this, consider a finite group $G$ acting linearly on a finite dimensional vector space $V$. In other words, $V$ is a finite dimensional representation of $G$. For any subgroup $H$ of $G$, we can consider $V$ as a representation of $H$ by restricting the action of $G$ to $H$. This is aptly named restriction. However, the more interesting question is whether it is also possible to go in the opposite direction. In other words, given a finite dimensional representation $V$ of $H$, can we get a representation of $G$? To do this, we use a construction called induction. Consider the space $$W = \bop_{g_iH} g_iV$$
where the direct sum is indexed over all cosets in $G/H$. The $g_i$ are chosen representatives of each of these cosets. Without going into too much detail, the $G$-action on $W$ is given by letting $G$ act on its cosets rather than $V$ itself, essentially permuting these isomorphic copies of $V$. We could also have defined this construction using the tensor product to avoid the condition that $G$ be finite. As we will see, these constructions are very similar to those that will be used in our definition of a $G$-Mackey functor.   

Mackey functors also arise rather naturally in algebraic topology as the stable homotopy "groups" of $G$-spectra, described nicely by Kristen Mazur \cite{Mazur}. In the case of ordinary stable homotopy theory, the stable homotopy groups of (non-equivariant) spectra land in $Ab$, the category of abelian groups. However, in equivariant stable homotopy theory, we consider $G$-spectra and require that the stable homotopy groups somehow encode information about the action of all subgroups of $G$ as well. Working out the details of this requirement, the stable homotopy groups of $G$-spectra are actually forced to be Mackey functors.
\begin{exmp}Let $S^0$ represent the equivariant sphere spectrum. Then the zeroth stable homotopy group of the $H$-fixed points of the equivariant sphere spectrum is $A(H)$ for all subgroups $H$ of $G$. Here $A(H)$ is the Burnside ring of $H$. More specifically,
$$\pi_0((S^0)^H) := [S^0 \wedge (G/H)_+,S^0]^G = [S^0, S^0]^H = \underline{A}(G/H) = A(H)$$
for all subgroups $H$ of $G$, where $G/H$ is a finite $G$-set and $[S^0,S^0]^H$ is the set of homotopy classes of $H$-equivariant maps $S^0 \rightarrow S^0$. As $H$ varies over all subgroups of $G$, we can fit the abelian groups $\underline{A}(G/H)$ together to obtain the Burnside Mackey functor $\underline{A}$ described in Example \ref{exmp: Burnside}.
\end{exmp} 

For those who are curious, the sources cited above give more thorough descriptions of these motivations coming from other areas of mathematics. However, for the context of this paper, these objects will be defined and discussed in a purely algebraic setting. 

We begin by providing a short discussion of Mackey functors along with several examples to develop some notation and intuition for these objects. We will also see how we can view Mackey functors as objects in a category, leading us into our main discussion on the box product, a symmetric monoidal product on the category of Mackey functors. Section \ref{sect: DefnBox} is devoted entirely to developing a concrete definition for the box product for the case of $G=C_p$ and provides a very constructive approach to this task. However, the downside to this approach is that it is notationally heavy. Thus, in Section \ref{sect: Computations}, we work through several computations of box products for $C_p$-Mackey functors to become more comfortable with the definition and how it works computationally. Perhaps the next natural question is what it means for an object to be invertible under this product and how many invertible objects there might be. The remainder of the paper is focused on answering this question. We develop several tools for doing so in Section \ref{sect: Invertible MF}, before concluding with a classification theorem for invertible $C_p$-Mackey functors. \\

\section{$G$-Mackey Functors} \label{sect: MF}
  
\subsection{Equivalent Definitions of a $G$-Mackey Functor} \label{subsect: Defns} Let $G$ be a finite abelian group. We first present two equivalent definitions of a $G$-Mackey functor (or just Mackey functor when the group $G$ is clear) for an arbitrary finite abelian group. In fact, the first definition, attributed to Dress \cite{ADress}, holds for any finite group $G$. After stating these definitions, we will quickly present a definition specialized to $G=C_p$, the cyclic group of order a prime $p$, as these groups will be the main focus of this paper.  

\begin{defn} \label{defn: Dress MF} A \textit{$G$-Mackey functor \underline{M}} consists of a pair of functors $M_*$ and $M^*$ from the category of finite $G$-sets, $Set^G$, to the category of abelian groups, $Ab$, that agree on objects (we can thus define $\underline{M}(X):= M^*(X) = M_*(X)$ for any $X \in Set^G$) and take disjoint unions to direct sums. $M_*$ is covariant and $M^*$ is contravariant, and for every pullback diagram in $Set^G$: 

\[\xymatrix{ A \ar[r]^\alpha \ar[d]_\beta & B \ar[d]^\gamma \\
C \ar[r]_\delta & D
}
\]
the following diagram commutes.
\[\xymatrix{ \underline{M}(A) \ar[r]^{M_*(\alpha)}  & \underline{M}(B) \\
\underline{M}(C) \ar[u]^{M^*(\beta)} \ar[r]_{M_*(\delta)} & \underline{M}(D)\ar[u]_{M^*(\gamma)}
}
\] \\
\noindent The contravariant functor will be referred to as \textit{restriction} and the covariant functor as \textit{transfer}.  
\end{defn}

Alternatively, recall that any finite $G$-set is a disjoint union of orbits $G/H$, for $H$ a subgroup of $G$. Then by additivity, to define a Mackey functor, it suffices to determine $\underline{M}(G/H)$ as $H$ varies over the subgroups of $G$ as well as the restriction and transfer maps between each of these abelian groups. We then have the following more constructive definition of a Mackey functor. 

\begin{defn} \label{defn: GMF}Let $G$ be a finite abelian group. A \textit{Mackey functor} $\underline{M}$ is a collection of abelian groups $\underline{M}(G/H)$, as $H$ ranges over the subgroups of $G$, each accompanied by maps $tr_K^H: \underline{M}(G/K) \rightarrow \underline{M}(G/H)$ and $res_K^H : \underline{M}(G/H) \rightarrow \underline{M}(G/K)$ for all subgroups $K$ of $H$ such that:

\begin{enumerate}
\item $tr_J^H = tr_K^H tr_J^K$ and $res_J^H = res_J^K res_K^H$ for all subgroups $J \subseteq K \subseteq H$

\item $tr_K^H(\gamma \cdot x) = tr_K^H(x)$ for all $x \in \underline{M}(G/K)$ and $\gamma \in W_H(K)$. 

\item $\gamma \cdot res_K^H(x) = res_K^H(x)$ for all $x \in \underline{M}(G/H)$ and $\gamma \in W_H(K)$

\item For all subgroups $J, K \subset H$, $res_K^H tr_K^J(x) = \sum_{\gamma \in W_H(K)} \gamma \cdot tr_{J \cap K}^K(x)$ for all $x \in \underline{M}(G/(J \cap K))$. 
\end{enumerate}
where $W_H(K)$ is the Weyl group, $N_H(K)/K$. Notice that for $G$ abelian, $W_H(K) = H/K$ so the Weyl action is induced by the automorphisms of $H/K$. 
\end{defn}

\noindent An explanation on the equivalence between these two definitions is given by both Th\'evenaz and Webb \cite{TW} and Mazur \cite{Mazur}. 

\subsection{$C_p$-Mackey Functors} \label{subsect: C_p MF} 
When $G = C_p$, $p$ prime, $G$ has only two subgroups, $\{e\}$ and $C_p$. Then much of this definition becomes superfluous and we can reduce it to the following.

\begin{defn} \label{defn: CpMF} Let $G = C_p$. A \textit{Mackey functor} $\underline{M}$ is the pair of abelian groups $\underline{M}(*) := \underline{M}(C_p/C_p)$ and $\underline{M}(C_p) := \underline{M}(C_p/e)$ accompanied by maps $tr_e^{C_p}: \underline{M}(C_p) \rightarrow \underline{M}(*)$ and $res_e^{C_p} : \underline{M}(*) \rightarrow \underline{M}(C_p)$ such that for all $x \in \underline{M}(*), y \in \underline{M}(C_p), \gamma \in C_p$,
\begin{enumerate}
\item $\gamma \cdot res_e^{C_p}(x) = res_e^{C_p}(x)$
\item $tr_e^{C_p}(\gamma \cdot y) = tr_e^{C_p}(y)$ 
\item $res_e^{C_p}tr_e^{C_p}(y) = \sum_{\gamma \in C_p} \gamma \cdot y$
\end{enumerate}
The group $\underline{M}(C_p)$ is equipped with an action of $C_p$, whereas $\underline{M}(*)$ has an action of $\{e\}$, which is ignored. 
\end{defn}

\begin{notn}When there is no ambiguity, we abbreviate $res_K^H$ and $tr_K^H$ as $res$ (or just $r$) and $tr$, respectively.
\end{notn}

\begin{rem}In Definition \ref{defn: GMF}, we require knowledge of the Weyl group and its action. In all relevant cases for this paper, we have $G=C_p$ and the Weyl group is either trivial or $C_p$ itself. Hence, explicit mention of the Weyl group is often not made, as in the above definition.
\end{rem}

A concise way to capture the information of Definition \ref{defn: CpMF} is with a Lewis diagram, first introduced by Gaunce Lewis \cite{GLewis}. For $G = C_p$, we can describe any $C_p$-Mackey functor by the following diagram.
\[\xymatrix{
\underline{M}(*) \ar@/_1pc/[d]_{r} \\
\underline{M}( C_p ) \ar@/_1pc/[u]_{tr} }\] 
The simplicity of this diagram relies heavily upon $C_p$ having only two subgroups. A more general treatment is given by Mazur \cite{Mazur}, although we will not need this.  

We can also consider the category of Mackey functors, denoted $\mathfrak{M}_G$. The objects are Mackey functors and the morphisms are given as follows. 

\begin{defn} \label{def: Morphisms}Let $\underline{M}, \underline{N}$ be $G$-Mackey functors for $G$ a finite abelian group. A \textit{morphism of Mackey functors} $\phi: \underline{M} \rightarrow \underline{N}$ is a collection of $W_G(H)$-equivariant group homomorphisms $\phi_H: \underline{M}(G/H) \rightarrow \underline{N}(G/H)$ for all subgroups of $H$ such that $res_K^H\phi_H = \phi_K res_K^H$ and $tr_K^H \phi_K = \phi_H tr_K^H$ for all subgroups $K$ of $H$. For $G=C_p$, we can describe $\phi$ by the following.
\[\xymatrixcolsep{5pc}\xymatrix{
\underline{M}(*) \ar[r]^{\phi_{C_p}} \ar@/_1pc/[d]_{r_M} & \underline{N}(*) \ar@/_1pc/[d]_{r_N} \\
\underline{M}(C_p) \ar[r]^{\phi_e} \ar@/_1pc/[u]_{tr_M} & \underline{N}(C_p) \ar@/_1pc/[u]_{tr_N}}\] \\
The morphism $\phi$ is an \textit{isomorphism of Mackey functors} if $\phi_H$ is a group isomorphism for all $H \subseteq G$. 
\end{defn} 

\subsection{Common Examples} \label{subsect: Common Examples}
Although we will describe all of the following examples more generally, for the sake of simplicity all Lewis diagrams will be presented only for the group $G=C_p$.

\begin{exmp}Fixed Point Mackey Functor \\ 
Let $M$ be a module over the group ring $\Z[G]$. For all subgroups $H$ of $G$, define $$\underline{M}(G/H) := M^H$$ 
where $M^H$ is the subgroup of $M$ fixed by $H$. Then for all subgroups $K$ of $H$, the restriction map $res_K^H: M^H \rightarrow M^K$ is given by inclusion of fixed points and the transfer map $tr_K^H: M^K \rightarrow M^H$ is given by $tr_K^H(x) = \sum_{\gamma \in W_H(K)} \gamma \cdot x$. The Lewis diagram is given below for $G=C_p$. 

\[\xymatrix{
x \ar@{|->}[d] & M^{C_p} \ar@/_1pc/[d]_{r} & \sum_{\gamma \in C_p} \gamma \cdot x\\
x & M \ar@/_1pc/[u]_{tr} & x \ar@{|->}[u] }\]

This is the only example for which we will check the properties of a Mackey functor. Let $K \subseteq H \subseteq G$. Let $x \in M^H$. Since $res_K^H(x)$ is the image of an $H$-fixed point under the inclusion map, it is still fixed by $W_H(K)=H/K$ in $M^K$, showing the first property. For the second property,

$$tr_K^H(h \cdot y) = \sum_{\gamma \in H/K} \gamma \cdot (h \cdot y) = \sum_{\gamma \in H/K} (\gamma h) \cdot y = \sum_{\gamma \in H/K} \gamma \cdot y = tr_K^H(y)$$\\
for all $h \in H/K$. Lastly, for any $x \in M^K$, we have

$$res_K^Htr_K^H(x) = res_K^H \Big(\sum_{\gamma \in H/K} \gamma \cdot x \Big) = \sum_{\gamma \in H/K} \gamma \cdot x$$
\end{exmp}

\begin{exmp}\label{exmp: Constant}Constant Mackey Functor, $\underline{\Z}$

The constant Mackey functor can be viewed as a special case of the fixed point Mackey functor above. Consider $\Z$ as a $\Z[G]$-module with trivial $G$-action. Then $\underline{\Z}(G/H) := \Z^H = \Z$ for all subgroups $H$ of $G$. All restriction maps must be given by the identity and for all subgroups $K$ of $H$, 

$$tr_K^H(x) = \sum_{\gamma \in H/K} \gamma \cdot x = \sum_{\gamma \in H/K} x = |H/K| x$$
We have the following Lewis diagram.

\[\xymatrix{
\Z \ar@/_1pc/[d]_{1} \\
\Z \ar@/_1pc/[u]_{p} }\]
\end{exmp}

\begin{exmp}Orbit Mackey Functor, $\hat{\underline{M}}$

The orbit Mackey functor is defined in a way dual to the fixed point Mackey functor. Let $M$ be a $\Z[G]$-module. For all subgroups $H \subseteq G$, define 

$$\hat{\underline{M}}(G/H) := M_H$$
where $M_H$ is the quotient of $M$ by the action of $H$. Then for all subgroups $K \subseteq H \subseteq G$ and for all $x \in \hat{\underline{M}}(G/H)$, the restriction maps are given by $res_K^H(x) = \sum_{\gamma \in W_K(H)} \gamma \cdot x$ and the transfer maps by the surjection $M_K \surj M_H$.
 
\[\xymatrix{
M_{C_p} \ar@/_1pc/[d]_{p} \\
M \ar@/_1pc/[u]_{1} }\]
\end{exmp}

\begin{exmp}Permutation Mackey Functor 

Let $S$ be a finite $G$-set. Denote by $\Z[S]$ the free abelian group generated by $S$. We can view $\Z[S]$ as a $\Z[G]$-module by the action of $G$ on $S$. Then the permutation Mackey functor is equivalent to the fixed point Mackey functor for $\Z[S]$. 

We determine this explicitly for $G = S = C_2$. Recall $\Z[C_2] = \{a + b \gamma : a,b \in \Z\}$, where $\gamma$ is the nontrivial element of $C_2$. $G$ acts on $\Z[C_2]$ by $\gamma \cdot (a+b \gamma) = a \gamma + b$, so that the fixed points are exactly those elements of the form $a+a\gamma$, $ a \in \Z$. Then the permutation Mackey functor is given by the following:
\[\xymatrixcolsep{.6pc}\xymatrix{
1+\gamma \ar@{|->}[d] & \Z \langle 1+\gamma \rangle \ar@/_1pc/[d]_{r} & 1+ \gamma & 1+ \gamma \\
1+ \gamma  & \Z[C_2] \ar@/_1pc/[u]_{tr} & 1 \ar@{|->}[u] & \gamma \ar@{|->}[u]}\]
\end{exmp}

\begin{exmp} \label{exmp: Burnside} Burnside Mackey Functor, $\underline{A}$

For reasons that will be further explained in Section \ref{subsect: Burnside}, the Burnside Mackey functor is of great importance. To understand this Mackey functor, we must first give the following definition, attributed to none other than Burnside \cite{Burn}.

\begin{defn}The \textit{Burnside ring of a group $G$}, $A(G)$, is the Grothendieck group of the abelian monoid of isomorphism classes of finite $G$ sets. Addition for the monoid is given by disjoint union and the ring structure is given by multiplication under Cartesian product. 
\end{defn}

We can now define $\underline{A}$ on objects by $\underline{A}(G/H):= A(H)$, the Burnside ring of $H$. Let $Set^H$ be the category of all finite $H$-sets and let

$$i_K^*: Set^H \rightarrow Set^K , \quad H \times_K (-): Set^K \rightarrow Set^H$$ \\
denote the forgetful functor and induction functor, respectively. This induction functor can be thought of as taking $|H/K|$ copies of the given $H$-set and letting $H$ act on these copies by permuting the indices. Recall that $H$ acts on any of its cosets by $h' \cdot hK = (h'h)K$ for $h, h' \in H$. Then define:

$$r_K^H([X]) := [i_K^*(X)], \quad tr_K^H([Y]) := [G \times_H Y]$$ \\
for all $[X] \in \underline{A}(G/H)$ and $[Y] \in \underline{A}(G/K)$. 

We shall compute this Mackey functor explicitly for $G= C_p$. For $H = \{e\}$, the isomorphism classes of finite $e$-sets are just finite sets. The Grothendieck group is generated by the isomorphism class of a single point $[e]$, denoted as 1, so that $A(H) = \Z\langle 1 \rangle.$ Now suppose $H=C_p$. There are now two isomorphism classes, those with trivial action, generated by $[e]$, and those with an action of $C_p$, generated by the set $[C_p]$. Then we write $A(H) = \Z \langle 1, [C_p] \rangle $. The restriction map is determined by forgetting the action of $C_p$ on a $C_p$ set. Essentially, this counts the number of points in the set, so that $1 \mapsto 1$ and $[C_p] \mapsto p$. The transfer map takes $p$ copies of the singleton $[e]$ and gives a $C_p$ action by permuting the copies. Then $1 \mapsto [C_p]$. We have the following Lewis diagram: 
 
\[\xymatrixcolsep{.6pc}\xymatrix{
1 \ar@{|->}[d] & [C_p] \ar@{|->}[d] & \Z \langle 1, [C_p] \rangle \ar@/_1pc/[d]_{r} & [C_p] \\
1 & p & \Z \langle 1 \rangle \ar@/_1pc/[u]_{tr} & 1 \ar@{|->}[u] }\]\\
\indent We can similarly define what Lewis called the \textit{twisted Burnside Mackey functor} $_d\underline{A}$ by letting $1 \in \underline{A}(*)$ restrict to any $d \in \Z$. These will become especially important in our discussion of invertible Mackey functors. 
\end{exmp}

\section{Definition of the Box Product} \label{sect: DefnBox}
\noindent The category of $G$-Mackey functors $\mathfrak{M}_G$ has a symmetric monoidal structure given by the box product,
$$\Box : \mathfrak{M}_G \times \mathfrak{M}_G \rightarrow \mathfrak{M}_G$$
As we will show in Section \ref{subsect: Burnside}, the box product for $C_2$-Mackey functors has unit given by the Burnside Mackey functor. Although the box product can be defined much more simply categorically in terms of left Kan extensions, this definition lacks the constructive nature of its algebraic counterpart and hence does not provide much information needed for computations. As our goal in the following sections is to compute several examples of the box product, we shall develop only the algebraic definition, originally detailed by Gaunce Lewis \cite{GLewis}.

\textbf{For the remainder of this paper, let} $\mathbf{G= C_p}$.  Let $\underline{M}$ and $\underline{N} \in \mathfrak{M}_G$. Using Lewis diagrams, we can represent the box product as below.
\[\xymatrix{
(\underline{M} \, \Box \, \underline{N}) (*)\ar@/_1pc/[d]_{res_e^{C_p}} \\
(\underline{M} \, \Box \, \underline{N}) (C_p) \ar@/_1pc/[u]_{tr_e^{C_p}} }\] \\
As the box product is an analog to the tensor product in the category of Mackey functors, an initial attempt to describe the box product might yield the following definition: 

$$(\underline{M} \, \Box \, \underline{N}) (G/H) \overset{?}{=} \underline{M}(G/H) \otimes \underline{N}(G/H)$$\\
Recall that it suffices to define a Mackey functor only on these orbits. As the restriction map is supported by this definition, the above is sufficient for $G/H = C_p/e$. However,  the transfer map is not and thus we must modify this definition for $(\underline{M} \, \Box \, \underline{N}) (*)$. A natural solution would be to rather artificially introduce all transfers, and this is what we do, yielding 

\begin{equation}(\underline{M}(*) \otimes \underline{N}(*)) \oplus Im(tr)
\end{equation}\\
where $Im(tr) = (\underline{M}(C_p) \otimes \underline{N}(C_p))/_{C_p}$. We quotient by the action of $C_p$ here to force $(\underline{M}\, \Box\, \underline{N})(*)$ to satisfy the following property of a Mackey functor: $tr(\gamma (a \otimes b)) = tr(a \otimes b)$ for all $a \in \underline{M}(C_p), b \in \underline{N}(C_p),\gamma \in C_p $. Here, $C_p$ acts diagonally, i.e. $\gamma (a \otimes b) = \gamma a \otimes \gamma b $. Additionally, since the elements of (2.1) are direct sums of tensor products, we expect them to act as such. In particular, we want that they satisfy something like the following:

$$x \otimes \sum_{\gamma \in C_p} \gamma \cdot y \overset{?}{=} \sum_{\gamma \in C_p} \gamma (x \otimes y)$$\\
However, in the context of Mackey functors, this summation now has a meaning (it is the transfer map). Consider for now the fixed point Mackey functor or any such that $r(x) = x$. Then for $x \in \underline{M}(*)$, $y \in \underline{M}(C_p)$, 

\begin{align*}
tr(r(x) \otimes y) 
= \sum_{\gamma \in C_p} \gamma(r(x) \otimes y) 
= \sum_{\gamma \in C_p} \gamma r(x) \otimes \gamma y 
= x \otimes \sum_{\gamma \in C_p} \gamma y 
= x \otimes tr(y)
\end{align*}\\
This gives the relation $x \otimes tr(y) \sim tr(r(x) \otimes y)$ for all $x \in \underline{M}(*), y \in \underline{M}(C_p)$. We can do this similarly for $(\sum_{\gamma \in C_p} \gamma \cdot x) \otimes y$. By requiring all Mackey functors to satisfy this desired property, we have the following relations, called \textit{Frobenius reciprocity}. 

$$x \otimes tr(y) \sim tr(r(x) \otimes y) , \quad tr(x) \otimes y \sim tr(x \otimes r(y))$$\\
We then arrive at the final definition of the box product for $G= C_p$.

\begin{defn} Let $\underline{M}$ and $\underline{N} \in \mathfrak{M}_{C_p}$. Define
$$(\underline{M} \, \Box \, \underline{N}) (*) :=  ((\underline{M}(*) \otimes \underline{N} (*)) \oplus (\underline{M}(C_p) \otimes \underline{N}(C_p))/_{C_p})/_\sim$$
$$(\underline{M} \, \Box \, \underline{N}) (C_p) := \underline{M}(C_p) \otimes \underline{N}(C_p)$$
The transfer map $tr_e^{C_p}: (\underline{M} \, \Box \, \underline{N}) (C_p)  \rightarrow (\underline{M} \, \Box \, \underline{N}) (*)$  is subject to the relation $tr(\gamma \cdot x ) = tr(x)$ for all $x \in (\underline{M} \, \Box \, \underline{N}) (C_p)$. The $C_p$ action is given by $\gamma \cdot (a \otimes b) = \gamma a \otimes \gamma b $ for $\gamma \in C_p$ and the relations $\sim$ are given by:
\begin{align*}
x \otimes tr(y) &\sim tr(r(x) \otimes y)\\
tr(x) \otimes  y &\sim tr(x \otimes r(y))
\end{align*}
We define the restriction map $res_e^{C_p}:(\underline{M} \, \Box \, \underline{N}) (*)  \rightarrow (\underline{M} \, \Box \, \underline{N}) (C_p)$ by $res_e^{C_p}(x \otimes y) = r_{\underline{M}}(x) \otimes r_{\underline{N}}(y)$ and for all $tr_e^{C_p}(x) \in Im(tr_e^{C_p})$, 
$$res_e^{C_p}tr_e^{C_p}(x) = \sum_{\gamma \in C_p} \gamma \cdot x$$
Additionally, one can easily check bilinearity over scalars in $\Z$ for both transfer and restriction. 
\end{defn}
The above definition can be described by the following Lewis diagram:
\[\xymatrix{
( (\underline{M}(*) \, \otimes \, \underline{N} (*)) \oplus Im(tr))/ _{\sim} \ar@/_1pc/[d]_{res_e^{C_p}} \\
\underline{M}(C_p) \, \otimes \, \underline{N}(C_p) \ar@/_1pc/[u]_{tr_e^{C_p}} }\]

\section{Computations of Box Products} \label{sect: Computations}

\subsection{Involving The Constant Mackey Functor}

\begin{exmp}
We begin with a preliminary example, $\underline{\Z} \, \Box \, \underline{\Z}$, where $\underline{\Z}$ denotes the constant Mackey functor. Recall the copies of $\underline{\Z}$ can be represented by the following diagrams:

\[\xymatrix{
x \ar@{|->}[d]  &  \Z\langle x \rangle \ar@/_1pc/[d]_{r_1} & px & &  x' \ar@{|->}[d] & \Z\langle x' \rangle \ar@/_1pc/[d]_{r_2} & px' \\
y & \Z\langle y \rangle \ar@/_1pc/[u]_{tr_1} & y \ar@{|->}[u] & & y' & \Z\langle y' \rangle \ar@/_1pc/[u]_{tr_2} & y \ar@{|->}[u]}\] \\
Then $(\underline{\Z} \, \Box \, \underline{\Z}) (*) = \Z \langle x \otimes x' , \, tr(y \otimes y' ) \rangle /\sim$ and $(\underline{\Z} \, \Box \, \underline{\Z}) (C_p) = \Z \langle y \otimes y' \rangle$. As they stand, the generators coming from transfer are nothing more than symbolic filler. The objective now is to write the transfers in terms of the first set of generators using the relations we can obtain from Frobenius reciprocity. In general, there are several generators and this may be rather difficult. However, in this case, it is clear.

\begin{equation}
tr(y \otimes y') = tr(r_1(x) \otimes y')\sim x \otimes tr_2(y') = x \otimes px' = p(x \otimes x')
\end{equation} \\
Then the generator $tr(y \otimes y')$ was redundant and we can drop it from our generating set. Note that since there is trivial $C_p$ action on both $\Z\langle y \rangle$ and $\Z\langle y' \rangle$, then $\Z \langle y \otimes y' \rangle $ has trivial $C_p$ action as well. We are left with
\[\xymatrix{
\Z \langle x \otimes x' \rangle \ar@/_1pc/[d]_{r} \\
\Z \langle y \otimes y' \rangle \ar@/_1pc/[u]_{tr} }\] \\
Directly from the definition, we have $r(x \otimes x') = r_1(x) \otimes r_2(x') = y \otimes y'$. From (3.2), we immediately obtain $tr(y \otimes y') = p(x \otimes x')$. By the isomorphism, $x \mapsto x \otimes x'$, $y \mapsto y \otimes y'$, we now recognize the above as the constant Mackey functor, so that $\underline{\Z} \, \Box \, \underline{\Z} \simeq \underline{\Z}$.
\end{exmp}

\subsection{Involving The Burnside Mackey Functor} \label{subsect: Burnside}
We now aim to show that the unit for the box product is the Burnside Mackey functor $\underline{A}$. We begin with a specific example before showing the general case. First, recall the Lewis diagram for $\underline{A}$.

\[\xymatrixcolsep{.6pc}\xymatrix{
1 \ar@{|->}[d] & [C_p] \ar@{|->}[d] & \Z \langle 1, [C_p] \rangle \ar@/_1pc/[d]_{r} & [C_p] \\
1 & p & \Z \langle 1 \rangle \ar@/_1pc/[u]_{tr} & 1 \ar@{|->}[u] }\]

\begin{exmp} $\underline{A} \, \Box \, \underline{\Z}$

We shall use the above notation for $\underline{A}$ and represent $\underline{\Z}$ as in Example \ref{exmp: Constant}.
 
$$(\underline{A} \, \Box \, \underline{\Z}) (*) = \Z \langle 1 \otimes x , \, [C_p] \otimes x, \, tr(1 \otimes y) \rangle /\sim$$ 
$$(\underline{A} \, \Box \, \underline{\Z}) (C_p) = \Z \langle 1 \otimes y \rangle$$\\
Since the $C_p$ action is trivial, this gives no relations. We now determine those given by Frobenius reciprocity:

$$tr(1 \otimes y) = tr(r(1) \otimes y) \sim 1 \otimes tr(y) = 1 \otimes px = p(1 \otimes x)$$
$$
[C_p] \otimes x = tr(1) \otimes x \sim tr(1 \otimes r(x)) = tr(1 \otimes y)$$\\
Combining these, both $[C_p] \otimes x $ and $tr(1 \otimes y)$ may be eliminated from the generating set. This yields the Lewis diagram
\[\xymatrix{
\Z \langle 1 \otimes x \rangle \ar@/_1pc/[d]_{r} \\
\Z \langle 1 \otimes y \rangle \ar@/_1pc/[u]_{tr} }\] 
We can compute: $r(1 \otimes x) = r(1) \otimes r(x) = 1 \otimes y$ and $tr(1 \otimes y) = p(1 \otimes x)$. Then, again this is recognized to be the constant Mackey functor so that $\underline{A} \, \Box \, \underline{\Z} \simeq \underline{\Z}$ .
\end{exmp}

\begin{exmp}$\underline{A} \, \Box \, \underline{M}$\\
Let $\underline{M} \in \mathfrak{M}_G$. A generic $C_p$-Mackey functor has the following Lewis diagram. For simplicity, we will let the abelian groups be finitely generated. 

\[\xymatrix{
\underline{M}(*) = \Z \langle a_1, \dots, a_n \rangle \ar@/_1pc/[d]_{r_{\underline{M}}} \\
\underline{M}( C_p ) = \Z \langle b_1, \dots, b_m \rangle \ar@/_1pc/[u]_{tr_{\underline{M}}} }\] 

$$(\underline{A} \, \Box \, \underline{M}) (*) = \Z \langle 1 \otimes a_1 , \dots, 1 \otimes a_n, \, [C_p] \otimes a_1, \dots, [C_p] \otimes a_n, \, tr(1 \otimes b_1), \dots, tr(1 \otimes b_m) \rangle /\sim$$ 
$$(\underline{A} \, \Box \, \underline{M}) (C_p) = \Z \langle 1 \otimes b_1, \dots, 1 \otimes b_m \rangle$$

We have that the group action is given by $\gamma ( 1 \otimes b_i ) = \gamma 1 \otimes \gamma b_i = 1 \otimes \gamma b_i$ and so is completely determined by the action on $\underline{M}(C_p)$.

$$[C_p] \otimes a_i = tr(1) \otimes a_i \sim tr(1 \otimes r(a_i)) = tr \Big(1 \otimes \sum_{j=1}^m c_{ij} b_j \Big) = \sum_{j=1}^m c_{ij} tr(1 \otimes b_j)$$

$$tr(1 \otimes b_i) = tr(r(1) \otimes b_i) \sim 1 \otimes tr_{\underline{M}}(b_i) = 1 \otimes  \sum_{j=1}^n d_{ij} a_j = \sum_{j=1}^n d_{ij} (1 \otimes a_j) $$\\
for some $c_{ij}, d_{ij} \in  \Z$. Then all generators of the form $[C_p] \otimes a_i$ and $ tr(1 \otimes b_i)$ may be eliminated from the generating set and we have the following Lewis diagram.
\[\xymatrix{
1 \otimes a_i\ar@{|->}[d] & \Z \langle 1 \otimes a_1, \dots, 1 \otimes a_n \rangle \ar@/_1pc/[d]_{r} & 1 \otimes tr_{\underline{M}}(b_j) \\
1 \otimes r_{\underline{M}}(a_i) & \underline{M}( C_p ) = \Z \langle 1 \otimes b_1, \dots, 1 \otimes b_m \rangle \ar@/_1pc/[u]_{tr} & 1 \otimes b_j \ar@{|->}[u]}\] 
Then the restriction and transfer maps are both determined by the original restriction and transfer for $\underline{M}$ so that $\underline{A} \, \Box \, \underline{M} \simeq \underline{M}$ by the obvious isomorphism $1 \otimes a_i \mapsto a_i$, $1 \otimes b_i \mapsto b_i$.
 
\end{exmp}

\begin{exmp} \label{exmp: Twisted Burnside} Let $c,d \in \Z$. Then $_c\underline{A} \, \Box \, _d \underline{A} \simeq {_{cd} \underline{A}}$, by the following computation. Using the notation of Example \ref{exmp: Burnside} for the twisted Burnside Mackey functors, we have the following.

\[\xymatrixrowsep{.5pc}\xymatrix{& \Z \langle 1 \otimes 1, 1 \otimes [C_p], [C_p] \otimes 1, [C_p]\otimes [C_p], tr(1 \otimes 1) \rangle \ar@/_1pc/[dd]_{r} \\ 
_c\underline{A} \, \Box \, _d \underline{A} = &\\
& \Z \langle 1 \otimes 1 \rangle \ar@/_1pc/[uu]_{tr} }\]
We have the following relations. 
$$c \cdot tr(1 \otimes 1) = 1 \otimes [C_p]$$
$$p \cdot tr(1 \otimes 1) = [C_p] \otimes [C_p]$$
$$d \cdot tr(1 \otimes 1) = [C_p] \otimes 1$$
Then  $(_c\underline{A} \, \Box \, _d \underline{A})(*) = \Z \langle 1 \otimes 1, tr(1 \otimes 1) \rangle$. We then compute $r(1 \otimes 1) = c \otimes d = cd(1 \otimes 1)$ and $r(tr(1 \otimes 1)) = p(1 \otimes 1)$, so that the above Mackey functor diagram indeed represents $_{cd}\underline{A}$. This computation will be of great importance when determining invertible Mackey functors in Section \ref{sect: Invertible MF}. 
\end{exmp}

\subsection{Other Computations}
\begin{exmp} $\underline{\Z}[C_2] \, \Box \, \underline{\Z}[C_2] \simeq \underline{\Z}[C_2 \times C_2]$\\
We have the following set-up for this box product.

\[\xymatrixcolsep{.5pc}\xymatrix{
1+\gamma \ar@{|->}[d]  &  \Z\langle 1+\gamma \rangle \ar@/_1pc/[d]_{r_1} & 1+\gamma & 1+\gamma 
& & &  e + \alpha \ar@{|->}[d] & \Z\langle e+\alpha \rangle \ar@/_1pc/[d]_{r_2} & e+\alpha & e + \alpha \\
1+\gamma & \Z[C_2] \ar@/_1pc/[u]_{tr_1} & 1 \ar@{|->}[u] & \gamma \ar@{|->}[u] 
& & & e+\alpha & \Z[C_2] \ar@/_1pc/[u]_{tr_2} & e \ar@{|->}[u] & \alpha \ar@{|->}[u]}\]

$$(\underline{\Z}[C_2] \, \Box \, \underline{\Z}[C_2]) (*) = \Z \langle (1+\gamma) \otimes (e + \alpha), tr(1 \otimes e), tr(1 \otimes \alpha), tr(\gamma \otimes e), tr(\gamma \otimes \alpha) \rangle $$
$$(\underline{\Z}[C_2] \, \Box \, \underline{\Z}[C_2]) (C_2) = \Z \langle 1 \otimes e, 1 \otimes \alpha, \gamma \otimes e, \gamma \otimes \alpha \rangle $$
Notice that there is now an action of $C_2$ on $(\underline{\Z}[C_2] \, \Box \, \underline{\Z}[C_2]) (C_2)$ so that we have the following relations.
$$tr(1 \otimes e) = tr(\gamma (1 \otimes e)) = tr(\gamma \otimes \alpha)$$
$$tr(1 \otimes \alpha) = tr(\gamma (1 \otimes \alpha)) = tr(\gamma \otimes e)$$
Additionally, we have the following relation from Frobenius reciprocity.
$$(1+\gamma) \otimes (e+\alpha) = (1+\gamma) \otimes tr(e) = tr(r(1 + \gamma) \otimes e) = tr(1 \otimes e) + tr(\gamma \otimes e)$$
Then $(\underline{\Z}[C_2] \, \Box \, \underline{\Z}[C_2]) (*)$ reduces to $\Z \langle tr(1 \otimes e), tr(1 \otimes \alpha) \rangle $. Now observe that the restriction map is injective so that we can identify $tr(1 \otimes e)$ with its value under restriction. Recall that for this, we have the formula
$$r(tr(1 \otimes e)) = \sum_{\gamma \in C_2} (1 \otimes e) = 1 \otimes e + \gamma \otimes \alpha$$
Hence, $tr(1 \otimes e) = 1 \otimes e + \gamma \otimes \alpha$. We can perform the same computation to $1 \otimes \alpha$ to yield $tr(1 \otimes \alpha) =  1 \otimes \alpha + \gamma \otimes e$. Let $A = 1 \otimes e + \gamma \otimes \alpha$ and $B = 1 \otimes \alpha + \gamma \otimes e$. We now have the following Lewis diagram for $\underline{\Z}[C_2] \, \Box \, \underline{\Z}[C_2]$. 
\[\xymatrixcolsep{.6pc}\xymatrix{
A \ar@{|->}[d] & B \ar@{|->}[d] & \Z \langle A,B \rangle \ar@/_1pc/[d]_{r} & A & A & B & B \\
A & B & \Z\langle 1 \otimes e, 1 \otimes \alpha, \gamma \otimes e, \gamma \otimes \alpha \rangle  \ar@/_1pc/[u]_{tr} & 1\otimes e \ar@{|->}[u]  & \gamma \otimes \alpha \ar@{|->}[u] & 1\otimes \alpha \ar@{|->}[u] & \gamma \otimes e \ar@{|->}[u]}\]
Notice that the elements involving $A$ and those involving $B$ don't interact at all. This looks like ``two $\underline{\Z}[C_2]$ Mackey functors in one". And in some sense it is. It is worthwhile for the reader to convince themselves that this is in fact isomorphic to the permutation Mackey functor $\underline{\Z}[C_2 \times C_2]$. This generalizes to the following, which we shall not prove here. 
\end{exmp}

\begin{exmp}Let $X, Y$ be finite $C_p$-sets. Then $\underline{\Z}[X] \, \Box \, \underline{\Z}[Y] \simeq \underline{\Z}[X \times Y]$. 
\end{exmp}

\section{Invertible Mackey Functors} \label{sect: Invertible MF}
Here we provide the first nontrivial examples of Mackey functors invertible for the box product. These will be the main focus of the rest of the paper. For $G= C_p$, we show that these are precisely those Mackey functors $_d \underline{A}$ with $d \in \Z$ relatively prime to $p$. 

\begin{defn}Let $\underline{M} \in \mathfrak{M}_G$. We say that $\underline{M}$ is \textit{invertible for the box product} (or just invertible) if there exists an $\underline{N} \in \mathfrak{M}_G$ such that $\underline{M} \, \Box \, \, \underline{N}$ is isomorphic to the Burnside Mackey functor.  
\end{defn}

\subsection{Initial Examples}We begin by providing a proof of a result by Shulman \cite{Shulman}. 

\begin{lemma} \label{lemma: Twisted Iso} There is an isomorphism of twisted Mackey functors $_{c}\underline{A}$ and  ${_d \underline{A}} $ if and only if there is an $x \in \Z$ such that $c = \pm d+px$. 
\end{lemma}
\begin{proof}Suppose  $_{c}\underline{A} \simeq {_d \underline{A}}$. Then there are isomorphisms of abelian groups $\phi_{C_p}: \Z\langle 1, [C_p] \rangle  \rightarrow \Z\langle 1, [C_p] \rangle$ and $\phi_e: \Z \rightarrow \Z$. As the only group isomorphisms $\Z \rightarrow \Z$ are the identity and $1 \mapsto -1$, we can assume without loss of generality that $\phi_e = id$.  To determine $\phi_{C_p}$, we utilize the remaining properties of a morphism of Mackey functors. 
$$\phi_{C_p}([C_p]) = \phi_{C_p}(tr(1)) = tr(\phi_e(1)) = tr(1) = [C_p]$$
$$c = \phi_e(c) = \phi_e(r(1)) = r(\phi_{C_p}(1)) = r(y+x[C_p]) = dy + px$$
for some $x,y \in \Z$. Representing $\phi_{C_p}$ by a matrix $A$, these results imply
$$A = \begin{pmatrix}
y & x\\
0 & 1
\end{pmatrix}$$
Since $\phi_{C_p}$ is an isomorphism, $det(A) = y$ must be invertible in $\Z$. Then $y = \pm 1$ so that $c = \pm d +px$. \\ 
Now, suppose $c = \pm d +px$ and consider the following diagram. \[\xymatrixcolsep{5pc}\xymatrix{
\Z \langle 1, [C_p] \rangle \ar[r]^{{\pm 1 \, \, x}\choose{\,\,0 \,\,\,\,\, 1}} \ar@/_1pc/[d]_{{c}\choose{p}} & \Z \langle 1, [C_p] \rangle \ar@/_1pc/[d]_{{d}\choose{p}} \\
\Z \langle 1 \rangle  \ar@{=}[r] \ar@/_1pc/[u]_{(0 \,\, 1)} & \Z \langle 1 \rangle \ar@/_1pc/[u]_{(0 \,\, 1)}}\]  
We can routinely check that this is indeed a morphism of Mackey functors. To determine if it is an isomorphism, we first notice that the bottom is the identity and hence an isomorphism. Since the determinant of the matrix for $\phi_{C_p}$ is a unit in $\Z$, the top group homomorphism is also an isomorphism, as desired. \end{proof}

\begin{notn}Let $a,b \in \Z$. We write $(a,b)$ to denote the greatest common divisor of $a$ and $b$. 
\end{notn}

\begin{lemma} \label{lemma: Invertible Twisted} Let $_c\underline{A}$ be a twisted Burnside Mackey functor. Then there exists a $d \in \Z$ such that $_c\underline{A} \, \Box \, _d\underline{A}$ is isomorphic to $\underline{A}$ if and only if $(c,p)=1$.
\end{lemma} 
\begin{proof}
Recall from Example \ref{exmp: Twisted Burnside}, $_c\underline{A} \, \Box \, _d \underline{A} \simeq {_{cd}\underline{A}}$. By Lemma \ref{lemma: Twisted Iso}, $\underline{A} \simeq {_{cd}\underline{A}}$ if and only if there is an $x \in \Z$ such that $1 = \pm cd + px$, where $p$ and $c$ are both known. We know by the Bezout identity that this has integer solutions $d, x$ if and only if $p$ is relatively prime to $c$. However, we will quickly show this. The reverse direction follows from the extended Euclidean algorithm. Now suppose $(p,\pm c) \neq 1$. In this case, $p$ is prime so that we must have $c=pk$ for some $k \in \Z$. Then $1 = p(kd + x)$, a contradiction. 
\end{proof}

Put slightly differently, this tells us all twisted Burnside Mackey functors $_c\underline{A}$ with $(c,p) = 1$ are invertible, with inverse of the form $_d\underline{A}$ for some $d \in \Z$. Our goal now is to show that these are all invertible Mackey functors. However, first we will spend a bit of time determining all invertible $C_p$-modules, the reason for which will become immediately apparent in the following section.

\subsection{Invertible $C_p$-Modules}

\begin{defn}A left $C_p$-module is an abelian group $M$ with a left action of $C_p$ such that $\gamma(a+b) = \gamma(a) + \gamma(b)$ for all $\gamma \in C_p$ and $a,b \in M$. 
\end{defn} 

Notice that since $M$ is abelian, a right $C_p$-module is given by letting $a \cdot \gamma = \gamma^{-1} \cdot a$ for all $\gamma \in C_p$ and $a \in M$. Hence, we will simply refer to $C_p$-modules, rather than distinguishing between right and left. Then a $C_p$-module is a $\Z$-module (viewing $\Z$ as a ring) with a group action that is compatible with the module structure. Additionally, we can define a tensor product of $C_p$-modules. Let $M,N$ be $C_p$-modules. Forgetting the $C_p$ action, recall that $M$ and $N$ are $\Z$-modules. We obtain a new $\Z$-module by $M \otimes_\Z N$. To determine a $C_p$ action on $M \otimes_\Z N$ that is compatible with the $\Z$-module structure, we again do the most obvious thing by letting $C_p$ act diagonally. A $C_p$-module $M$ is called \textit{invertible} if there is a $C_p$-module $N$ such that $M \otimes N \simeq \Z$, where $\Z$ is given the trivial action. 

\begin{lemma} \label{lemma: Invertible AbGrp}Let $p$ be a prime integer. \\
(a) For $p$ odd, there is exactly one invertible $C_p$-module up to isomorphism. \\
(b) For $p=2$, there are exactly two invertible $C_p$-modules up to isomorphism. 
\end{lemma}
\begin{proof}(a)We first show that the only invertible $\Z$-module is $\Z$ itself and then consider the group action. Let $A \in Ab$ and suppose there is a $B \in Ab$ such that $A \otimes_\Z B \simeq \Z$. It is a fact we shall not prove here that any invertible $\Z$-module is finitely generated. Then by the classification theorem of finitely generated abelian groups, $A \simeq \Z^n \oplus (\bop_i \Z/{n_i}\Z)$ and $B \simeq \Z^m \oplus (\bop_j \Z/{m_j}\Z)$. Then we have:
\begin{align*}A \otimes_\Z B &\simeq (\Z^n \oplus (\bop_i \Z/{n_i}\Z) ) \otimes_\Z (\Z^m \oplus (\bop_j \Z/{m_j}\Z) )\\
&= \Z^{nm} \oplus \bop_j (\Z^n \otimes_\Z \Z/m_j\Z) \oplus \bop_i (\Z^m \otimes_\Z \Z/n_i\Z) \oplus \bop_{ij} (\Z/n_i\Z \otimes_\Z \Z/m_jZ)\
\end{align*}
Since $A \otimes_\Z B \simeq \Z$, the free summand must have $nm=1$, where $n,m \in \Z$ so that $n=m=1$. Then this reduces to the following. 

$$\Z \simeq \Z \oplus (\bop \Z/m_j\Z) \oplus (\bop \Z/n_i\Z) \oplus (\bop_{ij} \Z/(n_i,m_j)\Z) $$
Then evidently all $n_i = m_j = 0$ so that $A \simeq \Z \simeq B$.


We must now determine the $C_p$-action on $\Z$. There is a one to one correspondence between $C_p$-modules $M$ and group homomorphisms $C_p \rightarrow Aut(M)$, where $M$ is viewed as an abelian group and $Aut(M)$ denotes the set of group automorphisms on $M$. Letting $M = \Z$, recall that the set $Aut(\Z)$ has only two elements, the identity 1 and the automorphism determined by $1 \mapsto -1$, which we denote here as -1. Since $p$ is odd, we have only the trivial homomorphism sending all $\gamma \in C_p$ to 1. This corresponds to the $C_p$-module $\Z$ with trivial $C_p$-action, which is clearly invertible with inverse itself. Since this is the only possible $C_p$-module structure, this is all invertible $C_p$-modules for $p$ odd. \\
(b)When $p=2$, there is an additional non-trivial group homomorphism given by $1 \mapsto -1$, corresponding to the $C_p$-module $\Z_-$. Recall that this is $\Z$ equipped with the sign action. Notice that $\Z_-$ is also invertible with inverse itself. Since any group homomorphism is determined by where it sends $[1]$, this in fact exhausts all possible homomorphisms. Hence, there are exactly two invertible $C_2$ modules given by $\Z$ and $\Z_-$.   
\end{proof}

\subsection{Classification of Invertible Mackey Functors} Let $\underline{M} \in \mathfrak{M}_G$ be invertible. Then there is some $\underline{N} \in \mathfrak{M}_G$ such that the following is an isomorphism of Mackey functors. 

\[\xymatrix{
( \underline{M}(*) \, \otimes \, \underline{N} (*) \oplus Im(tr))/ _{\sim} \ar@/_1pc/[d]_{res_e^{C_p}} \ar[rr]^-{\phi_{e}} & & \Z \oplus \Z \ar@/_1pc/[d]_{{1}\choose{p}}\\
\underline{M}(C_p) \, \otimes \, \underline{N}(C_p) \ar@/_1pc/[u]_{tr_e^{C_p}} \ar[rr]^-{\phi_{C_p}} & & \Z \ar@/_1pc/[u]_{(0 \,\, 1)} }\] \\
From this, we can determine what each tier of $\underline{M}$ must look like. First notice that the above isomorphism of Mackey functors implies there is an isomorphism of abelian groups $\underline{M}(C_p) \, \otimes \, \underline{N}(C_p) \simeq \underline{A}(C_p)$ respecting the action of $C_p$. In other words, $\underline{M}(C_p) \, \otimes \, \underline{N}(C_p) \simeq \Z$ as $C_p$-modules. For now let us consider $p$ odd. By Lemma \ref{lemma: Invertible AbGrp}, we must have $\underline{M}(C_p) = \Z$. Then half of our work is already done! To determine what $(\underline{M} \, \Box \, \underline{N} )(*)$ must look like, we begin with the following result. 

\begin{defn}A $G$-Mackey functor $\underline{M}$ is \textit{torsion free} if $\underline{M}(G/H)$ is a torsion free abelian group for each subgroup $H$ of $G$. \end{defn}

\begin{prop} \label{prop: Torsion Free} Let $\underline{M} \in \mathfrak{M}_{C_p}$, $p$ odd, be invertible. Then $\underline{M}$ is torsion free. 
\end{prop} 

In order to prove this statement, we need to develop a tool that is a Mackey functor analog to the isotropy separation sequence in equivariant stable homotopy theory. This will ``separate" our Mackey functor into parts controlled by $\underline{M}(*)$ and those controlled by $\underline{M}(C_p)$. We begin by defining a few useful Mackey functors. With this motivation coming from another branch of mathematics, these seemingly simple algebraic objects will come with rather complicated geometric names. The name will be kept for completeness rather than direct relevance to the content of this paper.

\begin{defn}We can define the \textit{Borel nilpotent completion of \underline{M}} as follows.

\[ \xymatrixrowsep{.5pc} \xymatrixcolsep{1pc} \xymatrix {
& tr(\underline{M}(C_p)) \ar@/_1pc/[dd]\\
\Gamma_{C_p}(\underline{M}) := & \\
& \underline{M}(C_p) \ar@/_1pc/[uu]
}\]
The transfer map is that of $\underline{M}$ and restriction is given by the formula $res(tr(x)) = \sum_{\gamma \in C_p} \gamma \cdot x$. There is an obvious injection $\Gamma_{C_p}(\underline{M}) \hookrightarrow \underline{M}$ given from the group inclusion $tr(\underline{M}(C_p)) \subset \underline{M}(*)$.
\end{defn} 

\begin{lemma} \label{lemma: Gamma Torsion Free}Suppose $\underline{N} \in \mathfrak{M}_G$ is such that $\underline{N}(C_p)$ is torsion free and has trivial $C_p$-action. Then $\Gamma_{C_p}(\underline{N})$ is torsion free. 
\end{lemma}
\begin{proof}First suppose that $im(tr) \neq 0$ else the result is clear. Observe that given these conditions, $im(tr)$ must be torsion free. Else, there is some $x \neq 0 \in im(tr) \subset \underline{N}(*)$ that is $n$- torsion for some $n\in \N$. Then $0= res(nx) = n \cdot res(x) = n \cdot res(tr(y)) = \sum_{\gamma \in C_p} y = p y$ for some $y \in \underline{N}(C_p)$. However, then $y$ is $p$-torsion in $\Z$, a contradiction. Then $\Gamma_{C_p}(\underline{N})(*) = im(tr)$ and $\Gamma_{C_p}(\underline{N})(C_p) = \Z$ are both torsion free and hence $\Gamma_{C_p}(\underline{N})$ is.
\end{proof}

\begin{defn}Let $\Phi^{C_p}(\underline{M})$ be the cokernel of the inclusion $\Gamma_{C_p}(\underline{M}) \hookrightarrow \underline{M}$. These are called the \textit{geometric fixed points of \underline{M}}. 
\[ \xymatrixrowsep{.5pc} \xymatrixcolsep{1pc} \xymatrix {
& \underline{M}(*)/ tr(\underline{M}(C_p) \ar@/_1pc/[dd]\\
\Phi^{C_p}(\underline{M}) := & \\
& 0 \ar@/_1pc/[uu]
}\]
If the map $\underline{M} \rightarrow \Phi^{C_p}$ is an isomorphism, we say that $\underline{M}$ is \textit{geometric}. 
\end{defn}

Using these definitions, to any Mackey functor $\underline{M}$, we can associate the following short exact sequence, called the \textit{isotropy separation sequence}. 

$$0 \longrightarrow \Gamma_{C_p}(\underline{M}) \longrightarrow \underline{M} \longrightarrow \Phi^{C_p}(\underline{M}) \longrightarrow 0$$

\begin{exmp}For the Burnside Mackey Functor $\underline{A}$, the isotropy separation sequence is given by the following. 

\[\xymatrixrowsep{.5pc} \xymatrixcolsep{.5pc} \xymatrix {
& & \Z\langle [C_p] \rangle \ar@/_1pc/[dd] & & \Z \langle 1, [C_p]\ar@/_1pc/[dd] \rangle & & \Z\ar@/_1pc/[dd] & & \\
0 & \longrightarrow & & \longrightarrow & & \longrightarrow & & \longrightarrow & 0 \\
& & \Z \ar@/_1pc/[uu] & & \Z\ar@/_1pc/[uu] & & 0\ar@/_1pc/[uu] & & 
} \]
\end{exmp}

\begin{prop} \label{prop: GeomIso}
Let $\underline{M}, \underline{N} \in \mathfrak{M}_G$. There is a natural isomorphism 
$$\Phi^{C_p}(\underline{M} \, \Box \, \underline{N}) \cong \Phi^{C_p}(\underline{M}) \, \Box \, \Phi^{C_p}(\underline{N})$$
\end{prop}

\begin{proof}
We first compute the box product $\Phi^{C_p}(\underline{M}) \, \Box \, \Phi^{C_p}(\underline{N})$. This is much easier since these Mackey functors are geometric. 
\[ \xymatrixrowsep{.5pc} \xymatrixcolsep{.5pc}\xymatrix{
\underline{M}(*)/Im(tr_M) \ar@/_1pc/[dd] & & \underline{N}(*)/Im(tr_N) \ar@/_1pc/[dd] & & \underline{M}(*)/Im(tr_M) \otimes \underline{N}(*)/Im(tr_N) \ar@/_1pc/[dd] \\
& \Box & & = & \\
0 \ar@/_1pc/[uu] & & 0 \ar@/_1pc/[uu] & & 0 \ar@/_1pc/[uu]
}\]
Observe that there are no added relations from Frobenius reciprocity since in all cases the values of the transfers are zero. The computation for $ \Phi^{C_p}(\underline{M} \, \Box \, \underline{N})$ is more messy. Let $tr_\Box$ denote the transfer map of the box product. 
\[ \xymatrixrowsep{.5pc} \xymatrixcolsep{1pc} \xymatrix {
& (((\underline{M}(*) \otimes \underline{N}(*)) \oplus Im(tr_\Box))/\sim) / Im(tr_\Box) \ar@/_1pc/[dd]\\
\Phi^{C_p}(\underline{M} \, \Box \, \underline{N}) = &  \\
& 0 \ar@/_1pc/[uu] 
}\]
There is an obvious isomorphism of the bottom tiers. For the top, we must first understand the top tier of $\Phi^{C_p}(\underline{M} \, \Box \, \underline{N})$. Since all elements of $Im(tr_\Box)$ are now zero in the quotient, it may be tempting to simplify this to $\underline{M}(*) \otimes \underline{N}(*)$. However, Frobenius reciprocity often produces new relations between the generators of $\underline{M}(*) \otimes \underline{N}(*)$ and those of $Im(tr_\Box)$. We should consider instead the following. 
$$\underline{M}(*) \otimes \underline{N}(*)) \big / \langle a \otimes tr(y), tr(x) \otimes b \rangle$$
It is an exercise in algebra to check that this is in fact isomorphic to 
$$\underline{M}(*)/Im(tr_M) \otimes \underline{N}(*)/Im(tr_N)$$
This completes the proof.
\end{proof}

\begin{lemma}Let $\underline{M} \in \mathfrak{M}_G$ be invertible. Then $\Phi^{C_p}(\underline{M})(*) \simeq \Z$ as abelian groups. 
\end{lemma}
\begin{proof}Since $\underline{M}$ is invertible, there is some $\underline{N} \in \mathfrak{M}_G$ such that $\underline{M} \, \Box \, \underline{N} \simeq \underline{A}$. Then we have the following isomorphism by Proposition \ref{prop: GeomIso}. 
$$\Phi^{C_p}(\underline{A}) \cong \Phi^{C_p}(\underline{M}) \, \Box \, \Phi^{C_p}(\underline{N})$$
By definition, this implies there is an isomorphism of abelian groups $$\Z = \Phi^{C_p}(\underline{A})(*) \cong \Phi^{C_p}(\underline{M})(*) \otimes \Phi^{C_p}(\underline{N})(*)$$
Then $\Phi^{C_p}(\underline{M})(*)$ is invertible as an abelian group and hence isomorphic to $\Z$ by Lemma \ref{lemma: Invertible AbGrp}. 
\end{proof}

\begin{lemma}Let $\underline{M} \in \mathfrak{M}_{C_p}$ be invertible. Then $\Gamma_{C_p}(\underline{M}) \simeq \Z^k,$ $k \leq 1$. \end{lemma} 
\begin{proof} Since $\underline{M}$ is invertible, we have $\underline{M}(C_p) \simeq \Z$. Then we can apply Lemma \ref{lemma: Gamma Torsion Free} so that $tr(\underline{M}(C_p)) \simeq \Z^k$ for some $k \in \N$. Since $tr(\underline{M}(C_p))$ is the image of a homomorphism whose domain is $\underline{M}(C_p) = \Z$, we must necessarily have $k \leq 1$. 
\end{proof}

We now have all the tools needed for the proof of Proposition \ref{prop: Torsion Free}. 

\begin{proof} (of Proposition 4.8) Suppose $\underline{M} \in \mathfrak{M}_G$ is invertible. We know $\underline{M}(C_p) \simeq \Z$ so all we need to show is $\underline{M}(*)$ is torsion free. We can write down the isotropy separation sequence of $\underline{M}$, with $\Gamma_{C_p}(\underline{M})$ and $\Phi^{C_p}(\underline{M})$ torsion free. In fact, from the above results, we know exactly what these Mackey functors look like.

\[\xymatrixrowsep{.5pc} \xymatrixcolsep{.5pc} \xymatrix {
& & \Z^k \ar@/_1pc/[dd] & & \underline{M}(*) \ar@/_1pc/[dd] & & \Z\ar@/_1pc/[dd] & & \\
0 & \longrightarrow & & \longrightarrow & & \longrightarrow & & \longrightarrow & 0 \\
& & \Z \ar@/_1pc/[uu] & & \Z\ar@/_1pc/[uu] & & 0\ar@/_1pc/[uu] & & 
} \]
for $k=0,1$. From this exact sequence of Mackey functors, we obtain the following short exact sequence of abelian groups.

$$0 \longrightarrow \Z^k \longrightarrow \underline{M}(*) \longrightarrow \Z \longrightarrow 0$$
Since $\Z$ is free, this short exact sequence splits, so that $\underline{M}(*) \simeq \Z ^k\oplus \Z$, which is torsion free as desired.
\end{proof}

However, we can be even more specific when it comes to classifying the possibilities for $\underline{M}(*)$. In fact, up to isomorphism, there is only one. 

\begin{lemma}Let $\underline{M} \in \mathfrak{M}_{C_p}$, $p$ odd, be invertible. Then $\underline{M}(*) \simeq \Z \oplus \Z$. \end{lemma} 
\begin{proof} For $\underline{M}$ invertible, we have already shown that $\underline{M}(*) \simeq \Z^k \oplus \Z$, $k=0,1$. Recall that for $d$ relatively prime to $p$, we have that $_d \underline{A}$ is invertible with $_d \underline{A}(*) \simeq \Z \oplus \Z$. Then $k=1$ is realized. However, $k=0$ is not. To see this, suppose $\underline{M}$ is invertible with inverse $\underline{N}$ and $\underline{M}(*) \simeq \Z$. Then $im(tr) = 0$. We shall label all restriction and transfer maps as follows, for some $c,d_1,d_2 \in \Z$.
\[ \xymatrixrowsep{.5pc} \xymatrix{
\Z \langle x \rangle \ar@/_1pc/[dd]_c & & \Z\langle y_1,y_2 \rangle \ar@/_1pc/[dd]_{{d_1}\choose{d_2}} & & \Z \langle x \otimes y_1, x \otimes y_2, tr(1 \otimes 1)\rangle \ar@/_1pc/[dd]_{res} \\
& \Box & & = & \\
\Z \ar@/_1pc/[uu]_0 & & \Z \ar@/_1pc/[uu]_{(0 \,\, 1)} & & \Z \langle 1 \otimes 1 \rangle \ar@/_1pc/[uu]_{tr}
}\]
Notice that for the transfer map of $\underline{N}$, we have chosen $1 \mapsto y_2$. All computations that follow are analogous for $1 \mapsto y_1$. We now compute the relations given by Frobenius reciprocity.
$$c \cdot tr(1 \otimes 1) = tr(c \otimes 1) \sim x \otimes tr(1) = x \otimes y_2$$
$$d_1 \cdot tr(1 \otimes 1) =  tr(1 \otimes d_1) \sim tr(1) \otimes y_1 = 0$$
$$d_2 \cdot tr(1\otimes 1) =  tr(1 \otimes d_2) \sim tr(1) \otimes y_1 = 0$$
Then $(\underline{M} \, \Box \, \underline{N})(*) = \Z \langle x \otimes y_1 \rangle$ and cannot be isomorphic to $\underline{A}(*) \simeq \Z \oplus \Z$. Analogous computations hold for $\underline{N}(*) \simeq \Z$ so that there is no $\underline{N}$ such that $(\underline{M} \, \Box \, \underline{N}) \simeq \underline{A}$, contradicting that $\underline{M}$ is invertible. Then $k=1$, as desired.
\end{proof} 

We have now shown that for $p$ odd, all invertible $C_p$-Mackey functors $\underline{M}$ have the following form.
\[ \xymatrix{
\Z \oplus \Z \ar@/_1pc/[d]_{res}\\
\Z \ar@/_1pc/[u]_{tr}
}\]
All that remains to determine are the values of the restriction and transfer maps on each generator. By the splitting of the isotropy separation sequence, we know that one summand is generated by $tr(1)$. Then $r(tr(1)) = p$, giving the value on the restriction map on one of the generators. Letting the other generator map to any integer $d$, we see that $\underline{M}$ is isomorphic to the twisted Burnside Mackey functor $_d \underline{A}$. Since its inverse $\underline{N}$ is also invertible, it must be of the same form so that $\underline{N} \simeq {_{d'} \underline{A}}$ for some $d' \in \Z$. By Lemma \ref{lemma: Invertible Twisted}, we must have $d$ relatively prime to $p$. These results can be summarized by the following statement, which for now we have only shown for $p$ odd. However, we shall state it without this condition.  

\begin{prop} \label{prop: Invertible odds}Let $\underline{M} \in \mathfrak{M}_{C_p}$ be invertible. Then $\underline{M} \simeq {_d\underline{A}}$ for some $d \in \Z$ such that $(d,p)=1$. $\quad \quad \quad \quad \quad \quad \quad \quad \quad \quad \quad \quad \quad \quad \quad \quad \quad \quad \quad \quad \quad \quad \quad \quad \quad \quad \quad \,\,\,\,\, \blacksquare$ 
\end{prop} 

We now address the case of $p=2$. Although all of the above computations hold for $p=2$ as well, there is an additional case to consider for $\underline{M}(C_p) = \Z_-$. We begin with the following lemma.

\begin{lemma}Let $\underline{M} \in \mathfrak{M}_{C_2}$ and suppose $\underline{M}(C_2) = \Z_-$. Then $res(x) = 0$ for all $x \in \underline{M}(*)$. 
\end{lemma}
\begin{proof}Let $x \in \underline{M}(*)$. By the definition of a Mackey functor, $\gamma \cdot res(x) = res(x)$ for all $\gamma \in C_2$. Letting $\gamma$ be the nontrivial element of $C_2$, we have $-res(x) = res(x) \in \Z$. Then $res(x) = 0$. Since $x$ was arbitrary, the result follows.
\end{proof}

Now suppose $\underline{M}$ were invertible with inverse $\underline{N}$ and suppose $\underline{M}(C_p) = \Z_-$. Then $\Z_- \otimes \underline{N}(C_p) \simeq \Z$. Recall that the action of $C_2$ on $\Z_- \otimes \underline{N}(C_p)$ is the diagonal action. From this, it is easy to show that $\underline{N}(C_p) \simeq \Z_-$ so that $r_{\underline{N}} \equiv 0$ as well. 

For any $x \otimes y \in \underline{M}(*) \otimes \underline{N}(*)$, we have $r(x \otimes y) = r_{\underline{M}}(x) \otimes r_{\underline{N}}(y) = 0$. Then the only nonzero restrictions come from images of the transfer map. Let $k(1 \otimes 1) \in \underline{M}(C_p) \otimes \underline{N}(C_p) = \Z\langle 1 \otimes 1 \rangle $. We have 
$$r(tr(k(1 \otimes 1))) = k \, res(tr(1 \otimes 1)) = 2k(1 \otimes 1)$$
Then $im(res) \subseteq 2\Z \subset \Z$. However, $im(res_{\underline{A}}) = \Z$ so that $\underline{M} \, \Box \, \underline{N}$ cannot be isomorphic to $\underline{A}$. Then in fact, this case does not contribute any invertible Mackey functors so that Proposition \ref{prop: Invertible odds} does in fact hold for $p=2$. Combining the results of Proposition \ref{prop: Invertible odds} and Lemma \ref{lemma: Invertible Twisted}, we obtain the following classification of invertible $C_p$-Mackey functors.

\begin{thm} \label{theorem: classify} Let $\underline{M}$ be a $C_p$-Mackey functor for $p$ prime. Then $\underline{M}$ is invertible if and only if $\underline{M}$ is isomorphic to $_d\underline{A}$ for some $d \in \Z$ such that $(d,p)=1.$ $ \quad \quad \quad \blacksquare$\\
\end{thm}

\subsection*{Acknowledgments}It is a pleasure to thank my mentors, Peter May and Dylan Wilson, for introducing me to this topic and teaching me so much interesting mathematics. I greatly appreciate all the guidance and encouragement they have given over the course of the REU program. I would also like to thank May for organizing a remarkable program and for his continued dedication that makes it all possible.

\end{document}